\documentclass{amsart}

\usepackage{amssymb}

\usepackage{mathpazo,courier}
\usepackage[scaled]{helvet}
\usepackage[mathcal]{euscript} 

\usepackage[usenames,dvipsnames]{xcolor}
\usepackage[pagebackref,colorlinks,linkcolor=BrickRed,citecolor=OliveGreen,urlcolor=Blue,hypertexnames=true]{hyperref}

\usepackage{enumerate}

\usepackage{tikz}

\theoremstyle{definition}
\newtheorem{df}[subsection]{Definition}
\newtheorem{thm}[subsection]{Theorem}
\newtheorem{lem}[subsection]{Lemma}
\newtheorem{rem}[subsection]{Remark}
\newtheorem{con}[subsection]{Conjecture}
\newtheorem{cor}[subsection]{Corollary}
\newtheorem{exm}[subsection]{Example}
\newtheorem{prop}[subsection]{Proposition}
\newtheorem{prob}[subsection]{Problem}
\newtheorem{notation}[subsection]{Notation}

\newcommand\N{\mathbb N}
\newcommand\R{\mathbb R}
\newcommand\C{\mathbb C}

\newcommand\RX{\R[x]}

\DeclareMathOperator\supp{supp}
\DeclareMathOperator\MAP{MAP}

\newcommand\1{\mathbf 1}

\newcommand\I{\mathcal I}
\newcommand\B{\mathcal B}

\DeclareMathOperator\cl{cl}

\newcommand\symdiff{\mathbin\Delta}

\newcommand\la\lambda
\newcommand\al\alpha
\newcommand\be\beta
\newcommand\si\sigma
\newcommand\ep\varepsilon

\title{Amalgamation of real zero polynomials}

\author{David Sawall}
\email{david.sawall@uni-konstanz.de}
\author{Markus Schweighofer}
\email{markus.schweighofer@uni-konstanz.de}
\address{Fachbereich Mathematik und Statistik, Universität Konstanz, 78457 Konstanz, Germany}
\date{July 28, 2023}

\begin{document}

\begin{abstract}
With this article, we hope to launch the investigation of what we call the \emph{Real Zero Amalgamation Problem}. Whenever a polynomial arises from another polynomial by substituting zero for some of its variables, we call the second polynomial an \emph{extension} of the first one. The \emph{Real Zero Amalgamation Problem} asks when two (multivariate real) polynomials have a common extension (called \emph{amalgam}) that is a real zero polynomial. We show that the obvious necessary conditions are not sufficient. Our counterexample is derived in several steps from a counterexample to amalgamation of matroids by Poljak and Turzík.
On the positive side, we show that even a degree-preserving amalgamation is possible in three very special cases with three completely different techniques.
Finally, we conjecture that amalgamation is always possible in the case of two shared variables. The analogue in matroid theory is true by another work of Poljak and Turzík. This would imply a very weak form of the Generalized Lax Conjecture.
\end{abstract}

\maketitle


\section{Introduction}
Throughout the article, we denote by $\N$, $\N_0$, $\R$, $\R_{\ge0}$ and $\R_{>0}$ the set of positive and nonnegative integers, the real, real nonnegative and real
positive numbers, respectively.
We consider three tuples $x=(x_1,\ldots,x_\ell)$, $y=(y_1,\ldots,y_m)$ and $z=(z_1,\ldots,z_n)$ of $\ell+m+n$ distinct variables for some $\ell,m,n\in\N_0$.
If $\ell=1$, then we sometimes write $x$ instead of $x_1$, similarly for $y$ and $z$. Furthermore, we will often consider an additional variable $s$ or
$t$ when we consider univariate polynomials or when we homogenize multivariate polynomials.

Let $p\in\R[x,y]$ and $q\in\R[x,z]$ be polynomials. We call a polynomial $r\in\R[x,y,z]$ an \emph{amalgam} of $p$ and $q$ if $p=r(x,y,0)$ and $q=r(x,0,z)$.  It is trivial
that an amalgam of $p$ and $q$ exists if and only if the compatibility condition $p(x,0)=q(x,0)$ holds. We will however ask for amalgams having certain properties.
All the properties we will consider will be preserved by setting some of the variables to zero. Hence we will have to require in addition that $p$ and $q$ have this property.

Our main interest lies in the property of being a \emph{real zero polynomial} \cite{hv,vin} but on the way we will also deal with the property
of being \emph{stable} \cite{pem,wag}.
A real polynomial in $n$ variables is called \emph{real zero} if it does not vanish at the origin and
its restriction along any line in $\R^n$ through the origin (seen as a real univariate polynomial) is real-rooted (i.e. has no non-real complex roots).

In Version 2 (posted on March 1, 2020) of the preprint \cite{sc2}, the second author showed by three completely different techniques in the following three special cases
that real zero polynomials $p\in\R[x,y]$ and $q\in\R[x,z]$ with $p(x,0)=q(x,0)$
can be amalgamated by a real zero polynomial $r$ with the additional property that the (total) degree of $r$ does not exceed the maximum degree of $p$ and $q$:
\begin{enumerate}[(a)]
\item $\ell=0$, i.e., if there are no shared variables,
\item $\ell=m=n=1$, i.e., if each block of variables consists just of a single variable,
\item $p$ and $q$ are quadratic, i.e., of degree at most $2$.
\end{enumerate}
We decided to move this material over to this article.
For (a), we will use the theory of stability preservers of Borcea and Brändén \cite{bb1}, for (b) the solution to the Lax Conjecture by Helton and Vinnikov \cite{hv} and
for (c) positive semidefinite matrix completion of Grone, Johnson, de Sá and Wolkowicz \cite{gjsw}.

In the same preprint of the second author, it was conjectured that two (compatible) real zero polynomials can always be amalgamated (by a real zero polynomial).
The first author recently disproved this after discussions with the second author. This is the main result of this paper. Concretely,
let $e_3\in\R[x_1,\ldots,x_6,y]$ be the third elementary symmetric polynomial in the seven variables $x_1,\ldots,x_6,y$ (here $\ell:=6$ and $m:=n:=1$). Consider the polynomials
\[p:=e_3-yx_1x_4-yx_3x_6-yx_2x_5-x_1x_2x_3-x_4x_5x_6\in\R[x,y]\] and \[q:=e_3(x,z)-zx_1x_4-zx_2x_5-x_1x_2x_3-x_4x_5x_6\in\R[x,z].\]
We show that the polynomials $p(x+\1,y)$ and $q(x+\1,z)$ are real zero and do not have a real zero amalgam although
$p(x+\1,0)=q(x+\1,0)$ where $\1$ always denotes the all-ones vector of appropriate length. 

This counterexample will rely on matroid theory and the theory of stable polynomials. Namely, we will first consider the problem of amalgamating two matroids
\cite[Section 11.4]{oxl}. It is well known in matroid theory that this is not always possible. We will take a counterexample from Poljak and Turzík from \cite{pt2} and consider the bases
generating polynomials $p$ and $q$ of the matroids.

We will show that the considered polynomials are stable using a criterion from Wagner and Wei \cite{ww}. If there were a multi-affine homogeneous stable amalgam of $p$ and $q$ then by \cite[Theorem 7.1]{cosw} its support would be an amalgamating matroid of the matroids considered in \cite{pt2}. This will give us that there is no homogeneous stable amalgam of $p$ and $q$. We will then show that the matroids Poljak and Turzík considered do not even have an amalgam as delta-matroids in the sense of Bouchet \cite{bou1}. From this, we get that there is no multi-affine stable amalgam of our bases generating polynomials because the support of a multi-affine stable polynomial forms a delta-matroid \cite{bra1}. We even get that our polynomials do not have a stable amalgam by using that the multi-affine part of a stable polynomial is again stable. Lastly, we use the theory of real zero polynomials and the fact that $p$ and $q$ do not have a stable amalgam to show that $p(x+\1,y)$ and $q(x+\1,z)$ are real zero and do not have a real zero amalgamating polynomial.

Knowing now that real zero amalgamation does not work under the obvious minimal requirements (contrary to what the second author conjectured in Version 2 of \cite{sc2}), the question is
what could potentially be an additional hypothesis that always allows for real zero amalgamation. Although, we are not able to give an answer for now, we allow ourselves to speculate
and search again for inspiration in matroid theory. Poljak and Turzík introduce in their earlier work \cite{pt1} the notion of a \emph{sticky} matroid which is by definition a matroid
over which amalgamation is always possible. They showed that \emph{modular matroids} \cite[Section 6.9]{oxl} are sticky and conjectured the converse.  Very recently, Shin
was able to prove this conjecture which was known under the name of \emph{sticky matroid conjecture} \cite{shi}. This raises the question what would be the ``sticky real zero
polynomials''. It is easy to check that all matroids over a two-element ground set are modular. Hence for sure, a counterexample which is of the same nature of the one
we present here will not exist in the case of $\ell=2$ shared variables. We will conjecture therefore that real zero amalgamation over two shared variables is always possible.
By (Version 2 of) \cite{sc2} this would imply the following very weak form of the Generalized Lax Conjecture: Given a rigidly convex set (the closure of the
connected component of the origin in the real non-vanishing set of a real zero polynomial which is known to be always convex) and a union of finitely many planes through the origin
(a two-dimensional subspace), there is a spectrahedron containing the rigidly convex set and agreeing with it on the finite union. One can think of this as wrapping a rigidly
convex set into a spectrahedron and ``tying it with cords''. The Generalized Lax Conjecture is much stronger and even says that each rigidly convex set is a spectrahedron. 


\section{Real zero and stable polynomials}\label{stableRZHyp}

Our main interest lies in real zero polynomials. They have been introduced in \cite[Subsection 2.1]{hv} but are just a non-homogeneous version of the more popular
(homogeneous) hyperbolic polynomials going back to Gårding \cite{gar1} which originally arose in the study of partial differential equations and later got popular in
convex optimization \cite{ren}.
For the proofs, we will need in addition the more recently investigated (real) stable polynomials \cite{bb1,bb2,bb3,wag} having applications in dynamical systems,
probability theory \cite{pem} and computer science \cite{oss}.

\begin{df}\label{dfstable}
We call a univariate polynomial $f\in\R[t]$ \emph{real-rooted} if all its complex roots are real, i.e., it is non-zero and splits into linear factors in $\R[t]$
(in particular, the zero polynomial is not real-rooted). Now let
$p\in\R[x]$ be a multivariate polynomial. We call $p$
\begin{enumerate}[(a)]
\item \emph{real zero} if $p(ta)\in\R[t]$ is real-rooted for all $a\in\R^\ell$,
\item \emph{stable} if $p(ta+b)\in\R[t]$ is real-rooted for all $a\in\R_{>0}^\ell$ and $b\in\R^\ell$
\end{enumerate}
\end{df}

Since we speak only about \emph{real} stable polynomials, we chose a definition in (b) above that avoids talking about complex numbers. This definition is however equivalent to
the one given in \cite{bb1} (making use of the open upper half plane in the complex numbers) by \cite[Lemma 2.5]{bb1}. 
By considering $a=0$ in (a), it follows easily that real zero polynomials do not vanish at the origin. Also note that the zero polynomial is neither real zero nor stable according
to our definition (whereas the zero polynomial is declared to be stable by some other authors \cite{wag}).

\begin{rem}\label{stableobvious}
The following are obvious:
\begin{enumerate}[(a)]
\item A univariate polynomial is real-rooted if and only if it is stable.
\item Any product of real zero polynomials is again real zero.
\item Any product of stable polynomials is again stable.
\item If $p\in\RX$ is stable then so is $p(x+a)$ for all $a\in\R^\ell$.
\end{enumerate}
\end{rem}

\begin{exm}\label{quadraticrealzero}
\begin{enumerate}[(a)]
\item Let $p\in\R[x]$ be a quadratic real zero polynomial with $p(0)=1$. Then $p$ can be uniquely written as
\[p=x^TAx+b^Tx+1\] with a symmetric matrix $A\in\R^{\ell\times\ell}$ and a vector $b\in\R^\ell$. For $a\in\R^\ell$ the univariate
quadratic polynomial $p(ta)=a^TAat^2+b^Tat+1$ splits in $\R[t]$ if and only if its discriminant $(b^Ta)^2-4a^TAa=a^T(bb^T-4A)a$ is
is nonnegative. Hence $p$ is a real zero polynomial if and only if the symmetric matrix $bb^T-4A$ is positive semidefinite.
\item The quadratic polynomial $1-x_1^2-x_2^2$ is real zero but not stable.
\item The quadratic polynomial $x_1x_2-1$ is stable but not real zero.
\end{enumerate}
\end{exm}

The following is folklore and easy to prove but we record it here because of its importance:

\begin{prop}\label{detexample}
Let $A_1,\ldots,A_\ell\in\R^{d\times d}$ be symmetric matrices. Then the polynomial \[\det(I_d+x_1A_1+\ldots+x_\ell A_\ell)\] is real zero.
\end{prop}
\begin{proof}
This follows easily from the fact that all complex eigenvalues of real symmetric matrices are real and that a univariate polynomial is real-rooted if and only if its reciprocal is real-rooted.
\end{proof}

Helton and Vinnikov proved by deep methods that surprisingly for $\ell=2$ variables the following converse to Proposition \ref{detexample}
holds \cite{hv} (see also \cite[Subsection 5]{han} for a more algebraic but yet difficult proof).
For $\ell=1$, the corresponding statement is trivial and $A_1$ can even be chosen diagonal.
For $\ell>2$, the converse of Proposition \ref{detexample} in this sense is in general very far from being true \cite{bra2}.

\begin{thm}[Helton and Vinnikov]\label{vinnikov}
If $p\in\R[x_1,x_2]$ is a real zero polynomial of degree $d$ with $p(0)=1$, then there exist symmetric $A_1,A_2\in\R^{d\times d}$ such that
\[p=\det(I_d+x_1A_1+x_2A_2).\]
\end{thm}

A weaker version of this theorem where real symmetric matrices are replaced by (complex) hermitian matrices would be enough for our purposes.
This weaker version seems to be considerably easier to prove (see mainly \cite{gkvw}, also \cite{pv} and \cite[Section 7]{han}). Proposition \ref{detexample}
obviously also continues to hold for hermitian matrices instead of real symmetric ones. In general, we could work throughout the article with hermitian
matrices if we replaced certain integrals over the orthogonal group with the corresponding integral over the unitary group. For ease of notation, we restrict
our exposition to the real case.

For our counterexamples in Sections \ref{stablecounter} and \ref{realzerocounter} to stable and real zero amalgamation, we will need to consider multi-affine parts of stable polynomials in the following sense:

\begin{df}\label{dfsupp}
For $\al\in\N_0^\ell$, we denote
\[x^\al:=x_1^{\al_1}\dotsm x_\ell^{\al_\ell}.\]
For any polynomial
\[p=\sum_{\al\in\N_0^\ell}a_\al x^\al\]
where all but finitely many of the $a_\al\in\R$ are zero, we define its
\emph{multi-affine part}
\[\MAP(p):=\sum_{\al\in\{0,1\}^\ell}a_\al x^\al.\]
We call $p\in\R[x]$ \emph{multi-affine} if $p=\MAP(p)$. For a multi-affine $p\in\R[x]$, it will be convenient for us to define its \emph{support}
in a slightly unusual way as
\[\supp(p):=\{\{x_i\mid\al_i=1\}\mid\al\in\{0,1\}^\ell,a_\al\ne0\}
\]
\end{df}

The following theorem has first been observed by Choe, Oxley, Sokal and Wagner \cite[Proposition 4.17]{cosw} (where the setup is slightly different but can easily be
adapted). Alternatively, it follows easily from the theory of linear stability preservers \cite[Page 542]{bb1} due to Borcea and
Brändén \cite[Theorem 2.1]{bb1} (see also \cite[Theorem 5.2]{wag} and \cite[Theorem 1.1]{lea}).

\begin{thm}[Choe, Oxley, Sokal and Wagner]\label{mapthm}
The multi-affine part of a stable polynomial is again stable unless it is the zero polynomial.
\end{thm}

The following example shows that the multi-affine part of a real zero polynomial is in general not real zero.

\begin{exm} The polynomial $p:=(1+x_1+x_2)(1-x_1+2x_2)$ is a real zero polynomial but its multi-affine part $\MAP(p)=1+3x_2+x_1x_2$ is not real zero since the quadratic
univariate polynomial $p(3t,t)=1+3t+3t^2\in\R[t]$ has negative discriminant $3^2-4\cdot 3\cdot 1$.
\end{exm}

The following definition stems from \cite{hv} but it is just a non-homogeneous version of the well-known notion of a hyperbolicity cone \cite{gar1,gar2,ren}.

\begin{df}
Let $p\in\RX$ be a real zero polynomial. The set \[C(p):=\{a\in\R^\ell\mid p(ta)\neq0\text{ for all }t\in(0,1)\}\] is called the \emph{rigidly convex} set of $p$.
\end{df}

It is well-known although not obvious that $C(p)$ is indeed convex \cite[Subsection 5.3]{hv} (a fact that we will crucially use in our counterexample \ref{RZEx} to real zero
amalgamation) and is the closure of the connected component at the origin of the complement of the real zero set of $p$ \cite[Subsection 2.2]{hv}. Moreover, even the following
is true and essentially already due to G\aa rding \cite{gar2}.

\begin{prop}[G\aa rding]\label{gar}\label{rzshift}
Let $p\in\RX$ be a real zero polynomial. Then $C(p)$ is convex and for each $a$ in the interior of $C(p)$, the shifted polynomial $p(x+a)$ is again a real zero polynomial
with $C(p)=\{a+b\mid b\in C(p(x+a))\}$.
\end{prop}

There are many connections and similarities between stable and real zero polynomials. As an example, we give those two that we will need. These following two
propositions can certainly be found elsewhere in disguised form or can be deduced easily from the literature. For convenience of the reader,
we state them in the form we need and give a proof.

\begin{prop}\label{shift}
Let $p\in\R[x]$ be homogeneous and stable and suppose $a\in\R_{\ge0}^\ell$ with $p(a)\ne0$. Then $p(a+x)$ is a real zero polynomial.
\end{prop}

\begin{proof}
Let $b\in\R^\ell$. We have to show that the univariate polynomial $f:=p(a+tb)$ is real-rooted. Setting $d:=\deg p\in\N_0$,
$f$ is real-rooted if and only if
\[g:=t^df\left(\frac1t\right)=t^dp\left(a+\frac bt\right)=p(ta+b)\]
is real-rooted (if $p(b)\ne0$, this is the reciprocal polynomial of $f$, otherwise it is the reciprocal polynomial times a power of $t$).
It is easy to see that there is a sequence $(a_n)_{n\in\N}$ in $\R_{>0}^\ell$ such that $\lim_{n\to\infty}a_n=a$ and $p(a_n)\ne0$ for all $n\in\N$.
By the stability of $p$, we have that $g_n:=p(ta_n+b)$ is then for each $n\in\N$ a real-rooted polynomial of degree $d$. By the continuity of roots \cite[Theorem 1.3.1]{rs}
it follows that $g=\lim_{n\to\infty}g_n$ is also a real-rooted polynomial.
\end{proof}

\begin{prop}\label{rzstable}
Let $p\in\R[x]$ be a real zero polynomial with $\R_{\geq0}^\ell\subseteq C(p)$. Then $p$ is stable.
\end{prop}

\begin{proof}
Let $a\in\R_{>0}^n$ and $b\in\R^n$. We claim that $p(ta+b)\in\R[t]$ is real-rooted. WLOG, we can suppose that $b\in\R_{>0}^\ell$
for otherwise we replace $b$ by $b+\la a$ for large enough $\la\in\R$. In particular, $b$ lies in the interior of $C(p)$. But then $p(x+b)$ is a real zero polynomial
by Proposition \ref{rzshift} and the claim follows.
\end{proof}

The aim of this article is to advertise the following problem and to give first positive and negative results on it:

\begin{prob}[Real Zero Amalgamation Problem]\label{rzap}
Suppose that $p\in\R[x,y]$ and $q\in\R[x,z]$ are real zero polynomials
with \[p(x,0)=q(x,0).\] When does there exist a real zero polynomial $r\in\R[x,y,z]$ such that \[r(x,y,0)=p\qquad\text{and}\qquad r(x,0,z)=q\qquad?\]
\end{prob}

When $r$ as in Problem \ref{rzap} exists, we call it an \emph{amalgam} of the real zero polynomials $p$ and $q$. Without the conditions that $p$ and $q$ are real zero
or without the \emph{compatibility condition} $p(x,0)=q(x,0)$ such an amalgam could obviously not exist. In March 2020, the second author released a preprint
(the second version of \cite{sc2}) where he conjectured that no other condition is needed. In Section \ref{realzerocounter} we will disprove this. In Section \ref{hope}, we will
motivate why it could possibly still be true for the case of $\ell=2$ shared variables $x_1$ and $x_2$. By \cite{sc2}, this would still imply a very weak form of the
long-standing Generalized Lax Conjecture \cite[Subsection 6.1]{hv} saying that each rigidly convex set is a spectrahedron, i.e., the solution set of a linear matrix inequality.
This weak form says the following: Given a rigidly convex set and finitely many planes through the origin (i.e., two-dimensional subspaces), one can find a spectrahedron containing the rigidly convex set that agrees with the rigidly convex set on the union of these planes.


\section{Some sufficient conditions for real zero amalgamation}\label{positive}

In this section, we prove some positive results concerning the Real Zero Amalgamation Problem \ref{rzap}. We start with some very special situations where amalgamation is obviously possible.

\begin{rem}\label{detamalgam}
The Real Zero Amalgamation Problem \ref{rzap} is of course solvable in the case where $p$ and $q$ have simultaneous determinantal representations
\begin{align*}
p&=\det(I_d+x_1A_1+\ldots+x_\ell A_\ell+y_1B_1+\ldots+y_mB_m)\qquad\text{and}\\
q&=\det(I_d+x_1A_1+\ldots+x_\ell A_\ell+z_1C_1+\ldots+z_nC_n)
\end{align*}
with symmetric matrices $A_i,B_j,C_k\in\C^{d\times d}$ (``simultaneous'' refers here to having the same $A_i$ for both $p$ and $q$)
since then
\[r:=\det(I_d+x_1A_1+\ldots+x_\ell A_\ell+y_1B_1+\ldots+y_mB_m+z_1C_1+\ldots+z_nC_n)\]
is a real zero polynomial (in fact, even of degree at most $d$) by
Proposition \ref{detexample}.
\end{rem}

The next observation is that the case $\ell=0$ is trivial.

\begin{rem}\label{amalgam0}
For $\ell=0$, i.e., in the absence of shared variables, the Real Zero Amalgamation Problem \ref{rzap} becomes trivial.
This is because essentially the product works, more exactly
\[r:=\frac{pq}{p(0)}=\frac{pq}{q(0)}\] is a real zero amalgam by Remark \ref{stableobvious}(b). Of course, the degree of the amalgam $r$ will here
in general exceed the degrees of $p$ and $q$.
\end{rem}

While we have absolutely no idea of how to generalize the trivial idea in Remark \ref{amalgam0} to the case $\ell>0$, we have a certain hope that the idea of the
degree preserving amalgamation of simultaneous determinantal representations from Remark \ref{detamalgam} could be adapted to a more general case, say
where $p(x,0)=q(x,0)$ has a determinantal representation
\[p(x,0)=q(x,0)=\det(I_d+x_1A_1+\ldots+x_\ell A_\ell)\]
with real symmetric $A_i\in\R^{d\times d}$ which is always the case for $\ell=2$ by the Helton-Vinnikov theorem \ref{vinnikov} after assuming WLOG
$p(0)=q(0)=1$.

We know however only how to carry out this degree-preserving adaption in the case $\ell=0$ which is already covered by the non-degree-preserving amalgamation in
Remark \ref{amalgam0}. The starting point is Remark \ref{detamalgamtwist} below. It is a variant of Remark \ref{detamalgam} which we do not know how to generalize to the case $\ell>0$. Before we state it, we introduce some notation.

\begin{notation}\label{partialnotation}
For any polynomial $p\in\R[x]$ and $i\in\{1,\ldots,\ell\}$, we denote by $\partial_{x_i}$ the linear operator on $\R[x]$ that maps a polynomial to its partial derivative
with respect to $x_i$. For $k\in\N_0$ and $i\in\{1,\ldots,\ell\}$, $\partial_{x_i}^k$ is the $k$-fold application of this differential operator.
\end{notation}

Unfortunately, we do not know how to generalize the approach of Remark \ref{detamalgamtwist} to $\ell>0$, i.e., to the case with shared variables.

\begin{rem}\label{detamalgamtwist}
If $\ell=0$ and $p$ and $q$ have determinantal representations
\begin{align*}
p&=\det(I_d+y_1B_1+\ldots+y_mB_m)\qquad\text{and}\\
q&=\det(I_d+z_1C_1+\ldots+z_nC_n)
\end{align*}
with symmetric matrices $B_j,C_k\in\C^{d\times d}$
then it will turn out that the polynomial
\begin{align*}
r:=&\int_{O_d}\det(I_d+y_1B_1+\ldots+y_mB_m+U^T(z_1C_1+\ldots+z_nC_n)U)\;dU\\
=&\int_{O_d}\det(I_d+U^T(y_1B_1+\ldots+y_mB_m)U+z_1C_1+\ldots+z_nC_n)\;dU\\
=&\int_{O_d}\int_{O_d}\det(I_d+U^T(y_1B_1+\ldots+y_mB_m)U+V^T(z_1C_1+\ldots+z_nC_n)V)\;dU\;dV
\end{align*}
(of degree at most $d$) is a real zero polynomial. See Corollary \ref{isrealzero} below.
The integrals here are all taken with respect to the Haar (probability) measure on the orthogonal group
$O_d$. These are integrals of vector-valued functions with values in a finite-dimensional subspace of the vector space of polynomials. Equivalently, you could say
that the integral has to be understood coefficient-wise. It is obvious that $r(y,0)=p$ and $r(0,z)=q$.

An obvious advantage of $r$ over the amalgamating polynomial
\[\det(I_d+y_1B_1+\ldots+y_mB_m+z_1C_1+\ldots+z_nC_n)\]
from Remark \ref{detamalgam} is that it depends less on the concrete determinantal representations of $p$ and $q$. Namely, $r$ will not change
if one exchanges simultaneously
all $B_i$ by $U^TB_iU$ and all $C_j$ by $V^TC_jV$ for some $U,V\in O_d$ (i.e., one takes conjugate determinantal representations of $p$ and $q$).

In fact, a closer look shows that $r$ depends really only on $d$ and the polynomials $p$ and $q$ rather than on their given determinantal representations.
Moreover, this is even true line by line.
To see this, we fix a direction $(b,c)\in\R^m\times\R^n$ and show that the univariate polynomial $r(tb,tc)$ depends only on $d$ and
the univariate polynomials $p(tb)$ and $q(tc)$. Indeed
$r(tb,tc)$ depends obviously only on the (real) eigenvalues of the real symmetric matrices
$b_1B_1+\ldots+b_mB_m$ and $c_1C_1+\ldots+c_nC_n$ and their multiplicities. But these eigenvalues and their multiplicities correspond to the
roots of the univariate polynomials $t^dp(-t^{-1}a)$ and $t^dq(-t^{-1}b)$ and their multiplicities.

Since $r$ depends now only on $d$, $p$ and $q$, the strategy will now be to make this dependance explicit and hope that it gives an idea of how to define $r$
in case that $p$ or $q$ does not have a determinantal representation as above. It will indeed turn out (see Theorem \ref{integralbecomesderivative} below)
that $r$ as above arises as follows: Let
$\tilde p:=s^dp(s^{-1}y)\in\R[s,y]$ and $\tilde q:=t^dp(t^{-1}z)\in\R[t,z]$ denote the degree $d$ homogenizations of $p$ and $q$ with respect to new
variables $s$ and $t$, respectively. Then it will turn out that
\[\tilde r:=\frac1{d!}\sum_{\substack{i,j\in\N_0\\i+j=d}}\partial_s^i\partial_t^j\tilde p\tilde q\in\R[s,t,y,z]\]
lies in $\R[s+t,y,z]$ and that $r$ arises from $\tilde r$ by setting ``the variable'' $s+t$ to $1$
(e.g., by substituting $1$ for one of $s$ and $t$ and $0$ for the other variable, or by
substituting $\frac12$ for each of $s$ and $t$). This way of getting $r$ from $d$, $p$ and $q$ will serve as our definition of the degree-preserving
amalgamation in the general case $\ell=0$ even if $p$ and $q$ are real zero polynomials that do not have determinantal representations as above.
See the proof of Theorem \ref{amalgamate}(a) below.
\end{rem}

The following lemma of Marcus, Spielman and Srivastava will be an important ingredient for the proof of Theorem \ref{integralbecomesderivative}
below. The original proof \cite[Theorem 2.11]{mss2} is quite lengthy. We gave a different and more direct proof in the first draft of this article.
We are very grateful to the anonymous referee who revealed that the statement simply follows immediately from polarization.

\begin{lem}[Marcus, Spielman and Srivastava]\label{walshreinterpret}
Consider the polynomial ring $\R[s,t]$ in two single variables $s$ and $t$.
For all $d\in\N_0$ and $a_1,\ldots,a_d,b_1,\ldots,b_d\in\C$, we have
\[\sum_{\substack{i,j\in\N_0\\i+j=d}}\partial_s^i\partial_t^j\prod_{\al=1}^d(s+a_\al)\prod_{\be=1}^d(t+b_\be)=\sum_{\si\in S_d}\prod_{i=1}^d(s+t+a_i+b_{\si(i)}).\]
\end{lem}

\begin{proof} Both sides of the claimed identity are obviously multi-affine and symmetric in the $a_i$ and also in the $b_j$. By polarization
(see for example \cite[Subsection 5.3]{pem} or \cite[Section 4]{wag}), one can therefore suppose that all of the $a_i$ are the same and
all of the $b_j$ are the same. But then an easy calculation shows that the left hand side is just a binomial expansion of the right hand side.
\end{proof}

A variant of the next lemma has been proven by Marcus, Spielman and Srivastava \cite[Theorem 4.8]{mss1}. Whereas we sum over the subgroup $S_d$ of $O_d$
consisting of the permutation matrices, they sum over a certain different finite subgroup of $O_d$ and they assume $A$ and $D$ symmetric (although they forgot
to state this) instead of $A$ arbitrary and $D$ diagonal. They need several pages to prove their result whereas we can provide a short proof for ours.

\begin{lem}\label{invariance}
Let $A,D,U\in\R^{d\times d}$ where $D$ is diagonal and $U$ is orthogonal. Then \[\sum_{P\in S_d}\det(A+P^TDP)=\sum_{P\in S_d}\det(A+U^TP^TDPU).\]
\end{lem}

\begin{proof}
Let $\la_1,\ldots,\la_d\in\R$ be the consecutive diagonal entries of $D$. For all $I\subseteq\{1,\ldots,d\}$ we write
$\la^I:=\prod_{i\in I}\la_i$ and $B_I$ for the matrix arising from a square matrix $B$ of size $d$
by deleting all rows and columns indexed by an element of $I$.
An easy calculation shows that the left hand side of our claimed equation equals
\begin{align*}
\sum_{P\in S_d}\det(PAP^T+D)&=\sum_{P\in S_d}\sum_{I\subseteq\{1,\ldots,d\}}(\det((PAP^T)_I))\la^I\\
&=\sum_{k=0}^d\sum_{\substack{I\subseteq\{1,\ldots,d\}\\\#I=k}}\sum_{P\in S_d}(\det((PAP^T)_{\{1,\ldots,k\}}))\la^I\\
&=\sum_{k=0}^d\left(\sum_{P\in S_d}\det((PAP^T)_{\{1,\ldots,k\}})\right)\sum_{\substack{I\subseteq\{1,\ldots,d\}\\\#I=k}}\la^I\\
&=\sum_{k=0}^dc_k\underbrace{\left(\sum_{\substack{I\subseteq\{1,\ldots,d\}\\\#I=k}}\det(A_I)\right)}_{=:a_k}\sum_{\substack{I\subseteq\{1,\ldots,d\}\\\#I=k}}\la^I.
\end{align*}
where $c_k$ is the number of permutations on $d$ objects that maps the first $k$ of these objects to $k$ prescribed other objects, i.e., $c_k:=k!(d-k)!$.
Now $a_k$ is the coefficient of $t^k$ in $\det(A+tI_d)$ (up to sign it is therefore the $k$-th coefficient of the characteristic polynomial of $A$).
If we exchange in this calculation $A$ by $UAU^T$, then $a_k$ will become the coefficient of $t^k$ in $\det(UAU^T+tI_d)=\det(A+tI_d)$ and will therefore not change.
\end{proof}

All we will need from the last lemma is the following immediate consequence. Just like the last lemma, this corollary has a certain analogue
in the work of Marcus, Spielman and Srivastava \cite[Theorem 4.2]{mss1}. The same authors give a much more general variant in their later
work \cite{mss2} where they prove
for example that $A$ can be allowed to be arbitrary instead of diagonal if one averages on the left hand side over all \emph{signed} permutation matrices
(in \cite{mss2} see Theorem 2.10 in connection with Lemma 2.6 and Corollary 2.7, compare also the proof of their Theorem 2.11, note also that the unitary
group $U_d$ can easily be replaced by the orthogonal group $O_d$). The latter result
implies obviously ours so that we attribute the result to them. Note however again that our derivation of the special case we need has really been
much shorter.

\begin{cor}[Marcus, Spielman and Srivastava]\label{perm-orth}
For all $A,D\in\R^{d\times d}$ with $D$ diagonal,
\[\frac1{d!}\sum_{P\in S_d}\det(A+P^TDP)=\int_{O_d}\det(A+U^TDU)dU.\]
\end{cor}

Now we can reprove a slight variation of \cite[Theorem 1.2]{mss2}. Our proof will use directly Corollary \ref{perm-orth} and Lemma \ref{walshreinterpret}
and therefore indirectly Lemma \ref{invariance}. All this was just to explain that we are implementing the ideas of Remark \ref{detamalgamtwist}.
This idea was to mimic the glueing of determinantal representations as in Remark \ref{detamalgam} also in some cases where these determinantal representations
do not exist by introducing the twist that consists in conjugating with an orthogonal matrix. As explained before Remark \ref{detamalgamtwist}, we do not know how to handle this in the case $\ell>0$ where $p$ and $q$ have shared variables. If we could find a variant of the following theorem for $\ell>0$, we suspect that it would lead to
other cases where real zero amalgamation is possible.

\begin{thm}[Marcus, Spielman and Srivastava]\label{integralbecomesderivative}
Let $B_1,\ldots,B_m,C_1,\ldots,C_n\in\R^{d\times d}$ be symmetric matrices. Set
\begin{align*}
p&:=\det(sI_d+y_1B_1+\ldots+y_mB_m)\in\R[s,y]\qquad\text{and}\\
q&:=\det(tI_d+z_1C_1+\ldots+z_nC_n)\in\R[t,z].
\end{align*}
Then
\begin{multline*}
\int_{O_d}\det((s+t)I_d+y_1B_1+\ldots+y_mB_m+U^T(z_1C_1+\ldots+z_nC_n)U)\;dU\\
=\frac1{d!}\sum_{\substack{i,j\in\N_0\\i+j=d}}\partial_s^i\partial_t^jpq\in\R[s,t,y,z].
\end{multline*}
\end{thm}

\begin{proof}
By arguing pointwise, one easily reduces to the following variant of our claim: Let $B,C\in\R^{d\times d}$ be symmetric. Set
\[p:=\det(sI_d+B)\in\R[s]\qquad\text{and}\qquad q:=\det(tI_d+C)\in\R[t].\]
Then we claim that
\[\int_{O_d}\det((s+t)I_d+B+U^TCU)\;dU\\
=\frac1{d!}\sum_{\substack{i,j\in\N_0\\i+j=d}}\partial_s^i\partial_t^jpq\in\R[s,t].\]
Here we can suppose WLOG $B$ and $C$ to be diagonal. Using Corollary \ref{perm-orth}, we can rewrite the left hand side of our claim
which then gets
\[\sum_{P\in S_d}\det((s+t)I_d+B+P^TCP)=\sum_{\substack{i,j\in\N_0\\i+j=d}}\partial_s^i\partial_t^jpq.\]
But this holds due to Lemma \ref{walshreinterpret} since $B$ and $C$ are both diagonal.
\end{proof}

This following lemma follows (just as Theorem \ref{mapthm} above) again easily from the theory of linear stability preservers \cite[Page 542]{bb1} due to Borcea and
Brändén \cite[Theorem 2.1]{bb1} (see also \cite[Theorem 5.2]{wag} and \cite[Theorem 1.1]{lea}).

\begin{lem}[Borcea and Brändén]\label{prep0amalgam}
If $p\in\R[s,t]$ is stable, then
\[\sum_{\substack{i,j\in\N_0\\i+j=d}}\partial_s^i\partial_t^jp\]
is again stable.
\end{lem}

\begin{proof}
See \cite[Theorem 3.4]{bb2} or \cite[Lemma 6.1]{bb2}.
\end{proof}

All what we will need from the last lemma is the following fact that was already known (in a disguised form) to Walsh more than a hundred years ago
(see the last footnote in \cite{wal}, see also \cite[Section 5.3]{rs}).

\begin{cor}[Walsh]\label{walsh}
If $p,q\in\R[s]$ are real-rooted, then
\[\sum_{\substack{i,j\in\N_0\\i+j=d}}\partial_s^i\partial_t^jp(s)q(t)\in\R[s+t]\]
is again real-rooted when viewed as a univariate polynomial in $s+t$.
\end{cor}

The form in which we will need this for Part (a) of Theorem \ref{amalgamate} below is as follows:

\begin{cor}\label{isrealzero}
Let $d\in\N_0$ and let $p\in\R[y]$ and $q\in\R[z]$ be real zero polynomials of degree at most $d$.
Denote by $\tilde p:=s^dp(s^{-1}y)\in\R[s,y]$ and $\tilde q:=t^dq(t^{-1}z)\in\R[t,z]$ their degree $d$
homogenizations.
Then
\[\tilde r:=\sum_{\substack{i,j\in\N_0\\i+j=d}}\partial_s^i\partial_t^j\tilde p\tilde q\in\R[s,t,y,z]
\]
lies in $\R[s+t,y,z]$ and becomes a real zero polynomial $r$ in $\R[y,z]$ after substituting $s+t$ by $1$.
\end{cor}

\begin{proof}
From Lemma \ref{walshreinterpret} it follows that $\tilde r\in\R[s+t,y,z]$. To prove that $r$ is a real zero polynomial, we fix
$b\in\R^m$ and $c\in\R^n$ and show that the univariate polynomial $r(tb,tc)\in\R[t]$ which is of degree at most $d$ is real-rooted.
But this is equivalent to $t^dr(t^{-1}b,t^{-1}c)\in\R[t]$ being real-rooted. But this latter polynomial equals
\[t^d\tilde r(1,0,t^{-1}b,t^{-1}c)=t^d\tilde r(0,1,t^{-1}b,t^{-1}c).\] Now observe that $\tilde r$ is homogeneous of degree $d$ since
$\tilde p\tilde q$ is homogeneous of degree $2d$. We therefore have
\[t^d\tilde r(0,1,t^{-1}b,t^{-1}c)=\tilde r(0,t,b,c)\]
and this polynomial is real-rooted as desired by Corollary \ref{walsh} applied to $\tilde p(s,b)$ and $\tilde q(t,c)$.
\end{proof}

We are now able to prove our three positive results on real zero amalgamation where we use for each case another non-trivial ingredient, namely
the theory of stability preservers, the Helton-Vinnikov theorem and positive semidefinite matrix completion.

\begin{thm}\label{amalgamate}
The Real Zero Amalgamation Problem \ref{rzap} is solvable, even in such a way that the degree of the amalgam $r$ does not exceed the maximum of the degrees of
$p$ and
$q$, in each of the following cases:
\begin{enumerate}[(a)]
\item $\ell=0$, i.e., if there are no shared variables,
\item $\ell=m=n=1$, i.e., if each block of variables consists just of a single variable,
\item for quadratic polynomials, i.e., if the degrees of $p$ and $q$ are at most $2$.
\end{enumerate}
\end{thm}

\begin{proof} Let $d\in\N_0$ and let $p\in\R[x,y]$ and $q\in\R[x,z]$ be real zero polynomials of degree at most $d$ such that $p(x,0)=q(x,0)$.
WLOG, we suppose that $p(0,0)=q(0,0)=1$. In each of the cases (a), (b) and (c) we have to show that there exists
$r\in\R[x,y,z]$ of degree at most $d$ such that $r(x,y,0)=p$ and $r(x,0,z)=q$. We proceed very differently in each of the cases.

\medskip (a)\quad Here we suppose that $\ell=0$, i.e., $p\in\R[y]$ and $q\in\R[z]$. We claim that $\frac1{d!}r$ where $r$ is defined
exactly as in Theorem \ref{isrealzero} does the job. First note that it is of degree at most $d$. Finally,
$r(y,0)=\tilde r(1,0,y,0)$ arises from
\[\sum_{\substack{i,j\in\N_0\\i+j=d}}\partial_s^i\partial_t^j\tilde p(s,y)\tilde q(t,0)\in\R[s,t,y]\]
by setting $s$ to $1$ and $t$ to $0$. Since $\tilde q(t,0)=t^d\tilde q(1,0)=t^dq(0)=t^d$, only one of the terms in this sum survives when $t$ is set to $0$, namely
$\partial_t^dt^d=d!$. Hence $\frac1{d!}r(y,0)=p$. Analogously, one proves that $r(0,z)=\tilde r(0,1,0,z)=d!q$ and hence $\frac1{d!}r(0,z)=q$.

\medskip (b)\quad Here we suppose that $\ell=m=n=1$, i.e., $x$, $y$ and $z$ are single variables.
By the Helton-Vinnikov theorem \ref{vinnikov}, we can choose hermitian matrices
$A,A',B,C\in\C^{d\times d}$ such that
\[p=\det(I_d+xA+yB)\qquad\text{and}\qquad q=\det(I_d+xA'+zC).\]
By conjugating each of $A$ and $B$ with a suitable unitary matrix, we can WLOG assume that $A$ is diagonal.
Conjugating it once more with a suitable permutation matrix, we can moreover suppose that the (real)
diagonal entries of $A$ are weakly increasing. In the same way, we may assume that also $A'$ is a diagonal matrix with weakly increasing
diagonal. The diagonal entries of $A$ can now be reconstructed from the polynomial $\det(I_d+xA)$ by looking at its degree, its roots and
the multiplicities of its roots. We proceed in completely the same manner with $A'$. Because of the polynomial identity
$\det(I_d+xA)=p(x,0)=q(x,0)=\det(I_d+xA')$, we thus get $A=A'$. Now $r:=\det(I_d+xA+yB+zC)$ is an amalgamation polynomial just like
in Remark \ref{detamalgam}.

\medskip (c)\quad The cases $d\in\{0,1\}$ are easy. Hence we suppose here that $d=2$.
We will use the theory of positive semidefinite matrix completion from \cite{gjsw}.
By Example \ref{quadraticrealzero}(a), there are
\begin{itemize}
\item symmetric matrices $A\in\R^{\ell\times\ell}$, $B\in\R^{m\times m}$ and $C\in\R^{n\times n}$,
\item matrices $E\in\R^{\ell\times m}$ and $F\in\R^{\ell\times n}$ and
\item vectors $a\in\R^\ell$, $b\in\R^m$ and $c\in\R^n$
\end{itemize}
such that
\begin{align*}
p&=\begin{pmatrix}x^T&y^T\end{pmatrix}\begin{pmatrix}A&E\\E^T&B\end{pmatrix}\begin{pmatrix}x\\y\end{pmatrix}
+\begin{pmatrix}a^T&b^T\end{pmatrix}\begin{pmatrix}x\\y\end{pmatrix}+1\qquad\text{and}\\
q&=\begin{pmatrix}x^T&z^T\end{pmatrix}\begin{pmatrix}A&F\\F^T&C\end{pmatrix}\begin{pmatrix}x\\z\end{pmatrix}
+\begin{pmatrix}a^T&c^T\end{pmatrix}\begin{pmatrix}x\\z\end{pmatrix}+1
\end{align*}
where both ``discriminants''
\begin{align*}
P&:=\begin{pmatrix}aa^T-4A&ab^T-4E\\ba^T-4E^T&bb^T-4B\end{pmatrix}=
\begin{pmatrix}a\\b\end{pmatrix}\begin{pmatrix}a^T&b^T\end{pmatrix}
-4\begin{pmatrix}A&E\\E^T&B\end{pmatrix}\in\R^{(\ell+m)\times(\ell+m)}\\\text{and}\\
Q&:=\begin{pmatrix}aa^T-4A&ac^T-4F\\ca^T-4F^T&cc^T-4C\end{pmatrix}
=\begin{pmatrix}a\\c\end{pmatrix}\begin{pmatrix}a^T&c^T\end{pmatrix}
-4\begin{pmatrix}A&F\\F^T&C\end{pmatrix}\in\R^{(\ell+n)\times(\ell+n)}
\end{align*}
are positive semidefinite. The task is to find a matrix $G\in\R^{m\times n}$ such that the quadratic polynomial
\[r:=\begin{pmatrix}x^T&y^T&z^T\end{pmatrix}\begin{pmatrix}A&E&F\\E^T&B&G\\F^T&G^T&C\end{pmatrix}\begin{pmatrix}x\\y\\z\end{pmatrix}
+\begin{pmatrix}a^T&b^T&c^T\end{pmatrix}\begin{pmatrix}x\\y\\z\end{pmatrix}+1\]
is a real zero polynomial, i.e., the ``discriminant'' of $r$
\[\begin{pmatrix}aa^T-4A&ab^T-4E&ac^T-4F\\ba^T-4E^T&bb^T-4B&bc^T-4G\\ca^T-4F^T&cb^T-4G^T&cc^T-4C\end{pmatrix}=
\begin{pmatrix}a\\b\\c\end{pmatrix}\begin{pmatrix}a^T&b^T&c^T\end{pmatrix}-4\begin{pmatrix}A&E&F\\E^T&B&G\\F^T&G^T&C\end{pmatrix}\]
is positive semidefinite. Since $a$ and $c$ are now fixed vectors, this amounts to filling the blocks marked by a question mark in the
``partial matrix''
\[\begin{pmatrix}aa^T-4A&ab^T-4E&ac^T-4F\\ba^T-4E^T&bb^T-4B&\text?\\ca^T-4F^T&\text?&cc^T-4C\end{pmatrix}\]
by real numbers so that one obtains a positive semidefinite matrix of size $\ell+m+n$. This is a positive semidefinite matrix
completion problem. The undirected graph $G$ with loops whose edges correspond to the known entries is
\[G=\{1,\ldots,\ell+m\}^2\cup(\{1,\ldots,\ell\}\cup\{\ell+m+1,\ldots,\ell+m+n\})^2\]
and is obtained by glueing together a complete graph on $\ell+m$ vertices with a complete graph on $\ell+n$ vertices along a
complete graph on $\ell$ vertices. It is an easy exercise to show that this graph is \emph{chordal} in the sense of \cite{gjsw}, i.e.,
each cycle consisting of at least four pairwise distinct nodes in this graph has a chord. By \cite[Theorem 7]{gjsw} it follows that our
matrix completion problem can be solved since the principal submatrices corresponding to cliques of the graph in the partial matrix
are all positive semidefinite. Indeed, each such submatrix is a 
principal submatrix of the discriminants $P$ and $Q$ of $p$ and $q$ which are both positive semidefinite.
\end{proof}


\section{Amalgamation of Matroids}

Matroids generalize the concept of linear independence. It turns out there is a notion of amalgamation of matroids that has been studied since the 1980s
\cite[Subsection 11.4]{oxl} and that will turn out to be related to our notion of amalgamation of real zero polynomials.
We will need only the very basics of matroid theory \cite[Sections 1.1--1.4]{oxl}. We do not assume the reader to be familiar with it and instead will recall
everything we need together with the corresponding references in Oxley's standard textbook \cite{oxl}. There are many equivalent ways of defining matroids.
We follow here \cite[Page 7]{oxl}:

\begin{df}\label{dfmatroid}
    Let $S$ be a finite set and $\I$ be set of subsets of $S$. The tuple $M=(S,\I)$ is called a \emph{matroid} on $S$ if
    \begin{enumerate}[(a)]
        \item $\emptyset\in\I$,
        \item if $I\in\I$ and $J\subseteq I$, then $J\in\I$ and
        \item if $I,J\in\I$ such that $\#I<\#J$, then there is $x\in J\setminus I$ such that $I\cup\{x\}\in\I$.
    \end{enumerate}
    The set $S$ is called the \emph{ground set} of $M$ and
    the elements of $\I$ are called the \emph{independent sets} of $M$. The maximal (with respect to inclusion) elements of $\I$ are called the \emph{bases} of $M$.
\end{df}

From (c) one sees immediately that all bases of a matroid have the same number of elements \cite[Lemma 1.2.1]{oxl}.
The independent sets of a matroid are of course exactly the subsets of its bases. Now we define the notion of amalgamation of matroids which will be crucial for us
\cite[Pages 20, 100, 101 and 436]{oxl}.

\begin{df}\label{dfmatroidamalgam}
    Let $M$ be a matroid on $S$.
    \begin{enumerate}[(a)]
    \item For a subset $A\subseteq S$, we denote by $M|A$ the matroid on $A$ whose independent sets are exactly the sets $I\subseteq A$
    which are independent sets of $M$. We call $M|A$ the \emph{restriction} of $M$ to $A$.
    \item For a subset $A\subseteq S$, we denote by $M/A$ the matroid on $S\setminus A$ whose independent sets are exactly the sets $I\subseteq S\setminus A$
    for which there exists a basis $B$ of $M|A$ such that $I\cup B$ is an independent set of $M$. We call $M/A$ the \emph{contraction} of $A$ from $M$.
     \item Let $M_1$ and $M_2$ be matroids on $S_1$ and $S_2$, respectively. If there is a matroid $M$ on some set containing $S_1\cup S_2$ such that $M|S_1=M_1$ and $M|S_2=M_2$ then $M$ is called an \emph{amalgam} of $M_1$ and $M_2$. Clearly, if there exists an amalgam then $M_1|(S_1\cap S_2)=M_2|(S_1\cap S_2)$ (see \cite[Page 436]{oxl}).
    \end{enumerate}
\end{df}

Amalgamation of matroids is known to be not always possible. See Example \ref{MatroidCounterEx} below. To understand this, we will need more notions:
One can define a matroid by specifying the set of bases instead of the set of independent sets \cite[Corollary 1.2.5]{oxl}:

\begin{prop}\label{matroidbybases}
A set $\B$ of subsets of $S$ is a basis of a matroid on $S$ if and only if
    \begin{enumerate}[(a)]
        \item $\B\ne\emptyset$ and
        \item if $B,C\in\B$ and $x\in B\setminus C$, there is  $y\in C\setminus B$ such that $(B\setminus\{x\})\cup\{y\}\in\B$.
     \end{enumerate}
\end{prop}

We will need the notions of rank and closure \cite[Pages 21 and 25]{oxl}:

\begin{df}\label{dfclosure}
 Let $M$ be a matroid on $S$. The size of the bases of $M$ is called \emph{rank} of $M$ and denoted by $r(M)$. 
 More generally, the rank of a subset $A$ of $S$ is $r(M|A)$ and will be denoted by $r(A)$.
 For a subset $A$ of $S$,
\[\cl(A):=\{x\in S\mid r(A\cup\{x\})=r(A)\}\]
is called the \emph{closure} or \emph{span} of $A$.
\end{df}

Finally, we record the following properties of the rank function \cite[Lemmata 1.3.1, 1.4.2 and 1.4.3]{oxl}.

\begin{prop}\label{rankproperties}
For any matroid $M$ on $S$. the following properties hold:
    \begin{enumerate}[(a)]
        \item $r(A)\leq\#A$ for all $A\subseteq S$,
        \item $r(A)\leq r(B)$ for all $A\subseteq B\subseteq S$,
        \item $r(A\cap B)+r(A\cup B)\leq r(A)+r(B)$ for all $A,B\subseteq S$,
        \item $r(\cl(A))=r(A)$ for all $A\subseteq S$,
        \item $\cl(A)\subseteq\cl(B)$ for all $A\subseteq B\subseteq S$.
    \end{enumerate}
    Condition (c) is referred to as \emph{submodularity} of the rank function.
\end{prop}

In 1982, Poljak and Turzík gave a counterexample \cite[Example 1]{pt2} 
for matroid amalgamation that will ultimately lead to our counterexamples to amalgamation of stable and real zero polynomials
in Sections \ref{stablecounter} and \ref{realzerocounter}. For convenience of the reader, we include it here.

\begin{exm}\label{MatroidCounterEx}
Consider the matroids $M_1$ on $\{x_1,\ldots,x_6,y\}$ and $M_2$ on $\{x_1,\ldots,x_6,z\}$
given by the following drawing in the sense that all three element
subsets should be bases except those for which the three elements lie on one of the drawn lines.

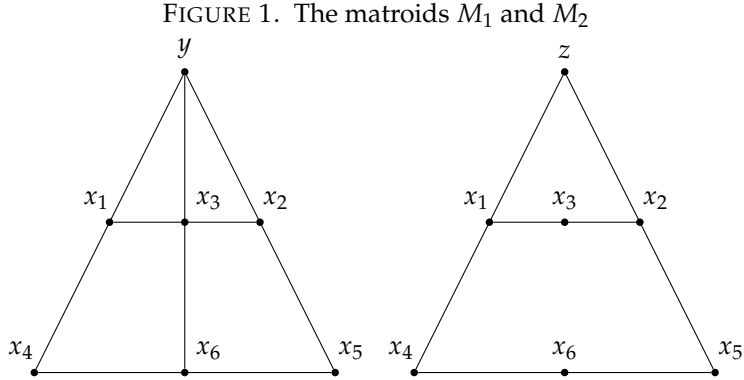
\begin{figure}[h]\label{figlabel}
\caption{The matroids $M_1$ and $M_2$}
\begin{tikzpicture}[scale=0.5]
\node[shape=circle,fill=black,scale=0.3,label=$y$] (0) at (4,8) {};
\node[shape=circle,fill=black,scale=0.3,label={[xshift=-.5em]$x_1$}] (1) at (2,4) {};
\node[shape=circle,fill=black,scale=0.3,label={[xshift=.9em]$x_3$}] (3) at (4,4) {};
\node[shape=circle,fill=black,scale=0.3,label={[xshift=.6em]$x_2$}] (2) at (6,4) {};
\node[shape=circle,fill=black,scale=0.3,label={[xshift=-.5em]$x_4$}] (4) at (0,0) {};
\node[shape=circle,fill=black,scale=0.3,label={[xshift=.9em]$x_6$}] (6) at (4,0) {};
\node[shape=circle,fill=black,scale=0.3,label={[xshift=.6em]$x_5$}] (5) at (8,0) {};

\path [-] (0) edge node[] {} (4);
\path [-] (0) edge node[] {} (6);
\path [-] (0) edge node[] {} (5);
\path [-] (4) edge node[] {} (5);
\path [-] (1) edge node[] {} (2);
\end{tikzpicture}
\begin{tikzpicture}[scale=0.5]
\node[shape=circle,fill=black,scale=0.3,label=$z$] (0) at (4,8) {};
\node[shape=circle,fill=black,scale=0.3,label={[xshift=-.5em]$x_1$}] (1) at (2,4) {};
\node[shape=circle,fill=black,scale=0.3,label=$x_3$] (3) at (4,4) {};
\node[shape=circle,fill=black,scale=0.3,label={[xshift=.6em]$x_2$}] (2) at (6,4) {};
\node[shape=circle,fill=black,scale=0.3,label={[xshift=-.5em]$x_4$}] (4) at (0,0) {};
\node[shape=circle,fill=black,scale=0.3,label=$x_6$] (6) at (4,0) {};
\node[shape=circle,fill=black,scale=0.3,label={[xshift=.6em]$x_5$}] (5) at (8,0) {};

\path [-] (0) edge node[] {} (4);
\path [-] (0) edge node[] {} (5);
\path [-] (4) edge node[] {} (5);
\path [-] (1) edge node[] {} (2);
\end{tikzpicture}
\end{figure}
Specifying matroids by such pictures is very common
but one usually has to check the matroid axioms \cite[Section 1.5]{oxl}. This is here easy (in fact these are even affine matroids \cite[Page 32]{oxl} which can be seen
by a slight horizontal shift of $x_3$ in the picture on the right hand side).
The restrictions of $M_1$ and $M_2$ to the set $\{x_1,\ldots,x_6\}$ agree.
We reproduce Poljak and Turzík's proof that there is no amalgam $M$ of $M_1$ and $M_2$.

Assume for a contradiction that there is an amalgam. Then denote by $r$ its rank function.
We consider the sets $A:=\{x_1,x_4\}$ and $B:=\{x_2,x_5\}$. We have $y,z\in\cl(A)\cap\cl(B)$. Thus, by Proposition \ref{rankproperties}, we get
\begin{align*}
        r(\{y,z\})+r(\cl(A)\cup\cl(B))&\leq r(\cl(A)\cap\cl(B))+r(\cl(A)\cup\cl(B))\\
        &\leq r(\cl(A))+r(\cl(B))=r(A)+r(B)=4.
    \end{align*}
    Since $r(\cl(A)\cup\cl(B))\geq3$, we get $r(\{y,z\})=1=r(\{y\})=r(\{z\})$. Thus
    \[2=r(\{y,x_3,x_6\})=r(\{y,z,x_3,x_6\})=r(\{z,x_3,x_6\})=3\]
    again by Proposition \ref{rankproperties}, which is impossible.
\end{exm}


\section{Counterexample for stable polynomials}\label{stablecounter}
In this section, we show that the bases generating polynomials of the two matroids from Example \ref{MatroidCounterEx} are stable polynomials and that there is no stable amalgam of these polynomials. We will use a criterion developed by Wagner and Wei in \cite{ww} to show that both bases generating polynomials are stable. 
Since the support of a homogeneous multi-affine stable polynomial corresponds to the set of bases of a matroid (see Theorem \ref{supportismatroid} below),
this will easily imply that there is no 
\emph{homogeneous} stable amalgam. To show however that there is no \emph{non-homogeneous} stable amalgam we will have to employ the
theory of delta-matroids emanating from the work of Bouchet \cite{bou1}.

\begin{df}
Let $M$ be a matroid on $\{x_1,\ldots,x_\ell\}$ where $x_1,\ldots,x_\ell$ are distinct variables. Let $\B$ be the set of bases of $M$.
We call the multi-affine homogeneous polynomial \[p_M:=\sum_{B\in\B}\prod_{v\in B}v\] the \emph{bases generating polynomial} of $M$.
\end{df}

Note that, in the situation of the above definition, the support of $p_M$ in the sense of Definition \ref{dfsupp} is $\B$.
We will now investigate how to get from such a bases generating polynomial the bases generating polynomial of certain restrictions and contractions in the sense
Definition \ref{dfmatroidamalgam} above. There is a pitfall here that is related to loops and coloops which we will now define.

\begin{df}
    Let $M$ be a matroid on $S$ and $x\in S$. Then $x$ is called a
    \begin{enumerate}[(a)]
    \item \emph{loop} if it is contained in no basis of $M$ \cite[Page 12, Exercise 2 of Section 1.2]{oxl} and
    \item \emph{coloop} if it is contained in all bases of $M$ \cite[Page 67]{oxl}.
    \end{enumerate}
\end{df}

The following is obvious and well-known \cite[Proposition 4.1]{cosw}. We will need only (a) and (b) but for completeness we also formulate (c) and (d). We write $p|_{x_i=a}$ for the polynomial arising from $p\in\R[x]$ by substituting $a\in\R$ for $x_i$. Recall also our Notation \ref{partialnotation} for partial derivatives.

\begin{rem}\label{loopremark}
Let $M$ be a matroid on $S:=\{x_1,\ldots,x_\ell\}$ where $x_1,\ldots,x_\ell$ are distinct variables.
Let $p:=p_M$ be its bases generating polynomial.
\begin{enumerate}[(a)]
\item If $v$ is not a coloop of $M$ then $p|_{v=0}$ is the bases generating polynomial of the restriction of $M$ to $S\setminus\{v\}$.
\item If $v$ is not a loop of $M$ then $\partial_vp$ is the bases generating polynomial of the contraction of $\{v\}$ from $M$.
\item If $v$ is a coloop of $M$ then $p|_{v=1}$ is the bases generating polynomial of the restriction of $M$ to $S\setminus\{v\}$.
\item If $v$ is a loop of $M$ then $p$ is the bases generating polynomial of the contraction of $\{v\}$ from $M$.
\end{enumerate}
\end{rem}

The following criterion from \cite[Theorem 3]{ww} will be very convenient for us.

\begin{thm}[Wagner and Wei]\label{WagnerWeiCharakterization}
    Let $p\in\R[x]$ be a multi-affine polynomial with only nonnegative coefficients. Then the following are equivalent:
    \begin{enumerate}[(a)]
        \item $p$ is stable,
        \item $\partial_{x_i}p$ and $p|_{x_i=0}$ are stable for all $i\in\{1,\ldots,\ell\}$ and there are $i,j\in\{1,\ldots,\ell\}$ with $i\ne j$ such that the
        (so-called Rayleigh) polynomial
        \[(\partial_{x_i}p)(\partial_{x_j}p)-p\partial_{x_i}\partial_{x_j}p\] is nonnegative in the whole of $\R^{\ell-2}$.
    \end{enumerate}
    Note that $(\partial_{x_i}p)(\partial_{x_j}p)-p\partial_{x_i}\partial_{x_j}p$ is actually a polynomial in $\ell-2$ variables.
\end{thm}

For the following theorem see \cite[Proposition 10.4]{cosw} or \cite[Page 602]{oxl}.

\begin{thm}[Choe, Oxley, Sokal and Wagner]\label{thm6}
    Let $M$ be a matroid on a set with at most $6$ elements. Then $p_M$ is stable.
\end{thm}

As a last preparation before we continue to study the implications of Example \ref{MatroidCounterEx} for stable polynomials, we state the following fundamental theorem
\cite[Theorem 7.1]{cosw}.

\begin{thm}[Choe, Oxley, Sokal and Wagner]\label{supportismatroid}
The support of a multi-affine homogeneous stable polynomial $p\in\R[x]$ is the set of bases of a matroid on $\{x_1,\ldots,x_\ell\}$.
\end{thm}

Now, we are ready to prove that the bases generating polynomials of $M_1$ and $M_2$ of Example \ref{MatroidCounterEx} are stable
and that they do not admit a \emph{homogeneous} stable amalgam. 

\begin{exm}\label{HomStableCounterEx}
We consider the polynomials $p:=p_{M_1}$ and $q:=p_{M_2}$, where $M_1$ and $M_2$ are the matroids from Example \ref{MatroidCounterEx}. In \cite[Corollary 8.2(b)]{cosw} it was shown that the matroid $M_1$ (denoted by $P_7$ in \cite[Appendix A.2.3]{cosw}) satisfies the so-called half-plane property, which is known to be equivalent to the stability of the bases generating polynomial of $M_1$ (see \cite[Corollary 5.14]{bra1} for the equivalence). In \cite[Page 1389]{ww} it was shown that the matroid $M_2$ (denoted by $\mathcal{P}_7'$ in \cite{ww} and by $P_7'$ in \cite{cosw}) has the half-plane property, i.e., $q$ is stable. For convenience of the reader, we present a slight variation of the argument of 
\cite{ww} for $M_2$ and use the same method for $M_1$ (which is completely different from the reasoning in \cite{cosw} for $M_1$). In both cases, we use crucially
Theorems \ref{WagnerWeiCharakterization} and \ref{thm6}. Since neither $M_1$ nor $M_2$ contains loops or coloops, we have that
$p$ and $q$ are stable after setting one of their variables to $0$ or after taking a partial derivative with respect to one of their variables by Remark \ref{loopremark}(a),(b)
and Theorem
\ref{thm6}. By Theorem  \ref{WagnerWeiCharakterization} it suffices thus to check that the two Rayleigh polynomials
$(\partial_{x_1}p)(\partial_{x_2}p)-p\partial_{x_1}\partial_{x_2}p$ and $(\partial_{x_1}q)(\partial_{x_2}q)-q\partial_{x_1}\partial_{x_2}q$
are nonnegative on the whole of $\R^5$. In fact, it happily turns out that they are even sums of squares of polynomials. Namely
using semidefinite programming and the Gram matrix method (see for example \cite{lau} or \cite[Section 2.6]{sc1}), we obtain
\begin{multline*}
	(\partial_{x_1}p)(\partial_{x_2}p)-p\partial_{x_1}\partial_{x_2}p=\\
	\left(yx_3+yx_6+x_3x_4+x_3x_5+x_3x_6+\frac{1}{2}x_4x_5+\frac{1}{2}x_4x_6+\frac{1}{2}x_5x_6\right)^2\\
	+\frac{3}{4}(x_4x_5+x_4x_6+x_5x_6)^2
\end{multline*}
and
\begin{multline*}
	(\partial_{x_1}q)(\partial_{x_2}q)-q\partial_{x_1}\partial_{x_2}q=\\
	\left(zx_3+\frac{1}{2}zx_6+x_3x_4+x_3x_5+x_3x_6+\frac{1}{2}x_4x_5+\frac{1}{2}x_4x_6+\frac{1}{2}x_5x_6\right)^2\\
	+\frac{1}{12}(3zx_6+x_4x_5+x_4x_6+x_5x_6)^2\\
	+\frac{2}{3}(x_4x_5+x_4x_6+x_5x_6)^2.
\end{multline*}
This shows that $p$ and $q$ are indeed stable.
    
Now, we show that $p$ and $q$ do not have a homogeneous stable amalgam. Assume for a contradiction that there is a homogeneous stable amalgam
$r\in\R[x,y,z]$ of $p$ and $q$. Using Theorem \ref{mapthm}, its multi-affine part $\MAP(r)$ is a stable, homogeneous and multi-affine amalgam of $p$ and $q$. Indeed,
\[(\MAP(r))(x,y,0)=\MAP(r(x,y,0))=\MAP(p)=p\]
since $p$ is multi-affine and evaluating at 0 and applying $\MAP$ commutes (both are just deletion of certain monomials) and similarly $(\MAP(r))(x,0,z)=q$. Thus, we can assume WLOG that $r$ is multi-affine. By Theorem \ref{supportismatroid}, the support of $r$ is now the set of bases of a matroid $N$ on the set $\{x_1,\ldots,x_6,y,z\}$. 
Since the monomials of $p$ and $q$ appear in $r$, every basis of $M_1$ and $M_2$ has to be a basis of $N$. Since $M_1$ and $M_2$ have no coloops, $N$ has thus
also no coloops. By Remark \ref{loopremark}(a), $N$ restricted to $\{x_1,\ldots,x_6,y\}$ is $M_1$ and $N$ restricted to $\{x_1,\ldots,x_6,z\}$ is $M_2$. By Example \ref{MatroidCounterEx} this is impossible. Hence, there is no homogeneous stable amalgam of $p$ and $q$.
\end{exm}

For later use, we record the following.

\begin{rem}\label{actually}
In the situation of Example \ref{HomStableCounterEx}, we have even shown that $p$ and $q$ have no (homogenous) multi-affine (not necessarily stable) amalgam $r\in\R[x,y,z]$ with the property that its support is the set of bases of a matroid on $\{x_1,\ldots,x_6,y,z\}$.
\end{rem}

For non-homogeneous stable polynomials, we need some of the theory of delta-matroids. Bouchet \cite[Page 156]{bou1} introduced delta-matroids to study greedy algorithms.
For any two sets $A$ and $B$, we denote by
\[A\symdiff B:=(A\setminus B)\cup(B\setminus A)\]
its \emph{symmetric difference}.

\begin{df}\label{dfdelta}
    Let $S$ be a finite set and $\mathcal F$ be a non-empty set of subsets of $S$. Then $M=(S,\mathcal F)$ is called a \emph{delta-matroid} on $S$ if for all $A,B\in\mathcal F$ and all $x\in A\symdiff B$ there is some $y\in A\symdiff B$ such that $A\symdiff\{x,y\}\in\mathcal F$.
\end{df}

We call this condition the \emph{symmetric exchange property}. Note that it perfectly allows for $y=x$ in which case $A\symdiff\{x,y\}=A\symdiff\{x\}$. Using Proposition \ref{matroidbybases}, the reader shows easily the following with an almost identical argument \cite[Page 64]{bou2}:

\begin{prop}[Bouchet]\label{lowerupper}
Let $M=(S,\mathcal F)$ be a delta-matroid.
\begin{enumerate}[(a)]
\item The minimal elements of $\mathcal F$ form a set of bases of a matroid that we call the \emph{lower matroid} of $M$.
\item The maximal elements of $\mathcal F$ form a set of bases of a matroid that we call the \emph{upper matroid} of $M$.
\end{enumerate}
\end{prop}

Whereas Part (a) of the last proposition will be crucially used in Example \ref{StableEx} below, we will need Part (b) only in the following proposition which we only include as an
additional motivation of the notion of delta-matroid. The result is from Bouchet \cite[Corollaries 7.3 and 7.4]{bou1} but we give a short self-contained proof.

\begin{prop}[Bouchet]
Let $S$ be a set and $\mathcal F$ be a set of subsets of $S$.
\begin{enumerate}[(a)]
\item $\mathcal F$ is the set of bases of a matroid if and only if $(S,\mathcal F)$ is a delta-matroid and all elements of $\mathcal F$ have the same cardinality.
\item $\mathcal F$ is the set of independent sets of a matroid if and only if $(S,\mathcal F)$ is a delta-matroid and $\mathcal F$ is closed under taking subsets.
\end{enumerate}
\end{prop}

\begin{proof}
The ``if'' part of (a) follows from either part of Proposition \ref{lowerupper} and the ``only if'' part from Proposition \ref{matroidbybases}.
The ``if'' part of (b) follows from Proposition \ref{lowerupper}(b). For the remaining part of (b) it suffices to show that every matroid in the sense of Definition \ref{dfmatroid}
is a delta-matroid. Hence let $M=(S,\I)$ be a matroid. Let $I,J\in\I$. We have to show that for each $x\in I\symdiff J$ there is some $y\in I\symdiff J$ such that $I\symdiff\{x,y\}\in\I$.
To this end, fix $x\in I\symdiff J$. By passing over from $M$ to its restriction $M|(I\cup\{x\})$ defined in Definition \ref{dfmatroidamalgam}(a)
(and exchanging $J$ by its intersection with $I\cup\{x\}$), we can suppose WLOG that
$I\cup\{x\}=S$. WLOG suppose $S\notin\I$. Then $I$ is a basis of $M$. Moreover the case where $x\in I\setminus J$ is trivial since it suffices then to take $y:=x$.
We finally treat the case where $x\in J\setminus I$. In this case we promise to find $y\in I\setminus J$ such that $I\symdiff\{x,y\}\in\I$. But then we can exchange $J$ by a basis of $M$ in which it is contained. Hence not only $I$ but also $J$ is WLOG a basis of $M$. But by the already proven Part (a) the set of bases of $M$ forms a delta-matroid on $S$.
So we know that there is $y\in I\symdiff J$ such that $I\symdiff\{x,y\}$ is a basis of $M$. It remains to show $y\in I\setminus J$.
But if we had $y\in J\setminus I$, then
the basis $I\cup\{x,y\}=I\symdiff\{x,y\}$ of $M$ would strictly contain the basis $I$ of $M$ which is impossible.
\end{proof}

Finally, we formulate the following result of Brändén which will be crucial to us \cite[Corollary 3.3]{bra1}.

\begin{thm}[Brändén]\label{isdelta}
The support of a multi-affine stable polynomial is a delta-matroid.
\end{thm}

Finally, we continue our Examples \ref{MatroidCounterEx} and \ref{HomStableCounterEx} and show that the stable analogue of Problem \ref{rzap} is in general not solvable.

\begin{exm}\label{StableEx}
Let again $M_1$ and $M_2$ be the matroids from Example \ref{MatroidCounterEx}. We show that there is no stable amalgam of $p:=p_{M_1}$ and $q:=p_{M_2}$.

Assume to the contrary that $r$ is such an amalgam. We seek for a contradiction. Exactly as in Example \ref{HomStableCounterEx}, we can suppose WLOG that $r$ is multi-affine
and therefore is of the form \[r=p+z\partial_zq+yzf=q+y\partial_yp+yzf\] for some multi-affine $f\in\R[x]$.
Thus, all monomials we add to the supports of $p$ and $q$ in order to amalgamate are multiples of $yz$ and thus of degree at least $3$ except for possibly
the monomial $yz$. By Theorem \ref{isdelta}, $\supp(r)$ is a delta-matroid on $S:=\{x_1,\ldots,x_6,y,z\}$.
Now we distinguish two cases and show that none of them can occur.

\smallskip
\textbf{Case 1:} $yz$ is a monomial of $r$.

\smallskip
Consider the monomial $x_1x_4x_5$, which appears in $p$ and thus in $r$. The monomials $x_1x_4x_5$ and $yz$ violate the symmetric exchange property. Indeed, after deleting $x_5$ from $\{x_1,x_4,x_5\}$, we can only remove $x_1$ or $x_4$ or add $y$ or $z$ or do nothing. But the monomials $x_4,\,x_1,\,yx_1x_4,\,zx_1x_4$ and $x_1x_4$ do not appear in $p$ or $q$ and thus do not appear in $r$. Hence, $yz$ cannot be a monomial of a stable multi-affine amalgam $r$.

\smallskip
\textbf{Case 2:} $yz$ is no monomial of $r$.

\smallskip
Then one deduces easily from Proposition \ref{lowerupper}(a) that the elements of the delta-matroid $\supp(r)$ of cardinality $3$ form the set of bases of a matroid $N$ on $S$.
Since $r$ is an amalgam of the cubic homogeneous polynomials
$p$ and $q$, so is its cubic homogeneous part $r_3$. Now $\supp(r_3)$
is the set of bases of $N$. By Remark \ref{actually} this is impossible.

\smallskip
In both cases, we attained the desired contradiction. Hence $p$ and $q$ have no stable amalgam.
\end{exm}

\begin{rem}
The lowest homogeneous part of a nonzero stable polynomial $p\in\R[x]$ is again stable (see \cite[Proposition 4.1]{rvy} and
\cite[Proposition 2.6]{ks}, the analogous fact for the highest homogeneous part was discovered already in \cite[Proposition 2.2]{cosw}). This can be seen by using Hurwitz's Theorem \cite[Theorem 1.3.8]{rs} and the fact that the lowest homogeneous part of $p$ is equal to
\[\lim_{\ep\to0}\ep^{-d}p(\ep x_1,\ldots,\ep x_\ell)\]
where $d$ is its degree.
For this reason, the polynomial $r_3$ in Case 2 of Example \ref{StableEx} is actually stable. Instead of appealing to Remark \ref{actually}, one could thus reduce in Case 2
directly to Example \ref{MatroidCounterEx}. We preferred our
argument via Remark \ref{actually} to keep the proof more self-contained and because we do not see how to avoid the use of delta-matroids in Case 1.

Indeed, for stable polynomials with nonnegative coefficients all of its homogeneous components are again stable \cite[Lemma 4.16]{bbl}. 
But this fails in general, e.g., the polynomial $(x-1)(y+1)$ is stable but its degree 1 component $x-y$ is not stable. Thus we cannot circumvent Case 1 by just taking the homogeneous component of degree 3 and neglecting the monomial $yz$.
\end{rem}


\section{Counterexample for real zero polynomials}\label{realzerocounter}

In the second version of the preprint \cite{sc1} released in March 2020, the second author conjectured that the Real Zero Amalgamation Problem \ref{rzap} would always admit a solution.
Using Example \ref{StableEx} and the theory of stable and real zero polynomials from Section \ref{stableRZHyp}, we are now ready to give the first counterexample.
For linear polynomials, it is trivially solvable. For quadratic polynomials we have shown in
Theorem \ref{amalgamate}(c) that it is still possible. Hence the smallest degree where it can fail is $3$. Indeed, we will provide cubic real zero polynomials that cannot
be amalgamated by an (arbitrary degree) real zero polynomial. By Remark \ref{amalgam0}, they will of course need to have $\ell\ge1$ shared variables (for $\ell=0$ there
would be even a degree-preserving real zero amalgam by Theorem \ref{amalgamate}(a)). In order to minimize the degrees of freedom for the amalgam, it is perhaps
not astonishing that the two blocks of non-shared variables will consist of just $m=n=1$ variable each. If we make this choice then we know by Theorem \ref{amalgamate}(b)
that we need to choose $\ell$ to be at least $2$. In Section \ref{hope} below, we will however argue that our particular technique based on matroids will not work for $\ell=2$.
Our example will have $\ell=6$ shared variables. We would suspect that counterexamples to real amalgamation with smaller $\ell$ can be found.

\begin{exm}\label{RZEx}
We consider the two cubic homogeneous polynomials $p:=p_{M_1}$ and $q:=p_{M_2}$ from the previous examples. In Example \ref{HomStableCounterEx}, we
recorded that $p$ and $q$ are stable.
Set $a:=(\mathbf 1,0)=(1,1,1,1,1,1,0)\in\R_{\ge0}^7$. Due to $p(a)=q(a)=18\neq0$, Proposition \ref{shift} yields
that $p(x+\1,y)$ and $q(x+\1,z)$ are real zero polynomials. Of course, we have $p(x+\1,0)=q(x+\1,0)$ since $p(x,0)=q(x,0)$.

Now suppose that there is a real zero polynomial $r\in\R[x,y,z]$ such that $p(x+\1,y)=r(x,y,0)$ and $q(x+\1,z)=r(x,0,z)$. Since the coefficients of $p(x+\1,y)$ and $q(x+\1,z)$ are nonnegative, we get that \[\R_{\geq0}^7\subseteq C(p(x+\1,y))\qquad\text{and}\qquad \R_{\geq0}^7\subseteq C(q(x+\1,z)).\] Thus, by the convexity of the rigidly convex set
$C(r)$ and the fact that $r$ is an amalgam of $p$ and $q$, we get that the convex hull of $(\R_{\geq0}^7\times\{0\})\cup(\R_{\geq0}^6\times\{0\}\times\R_{\geq0})$ is contained in
$C(r)$. But this convex hull is the orthant $\R_{\geq0}^8$ Thus, by Proposition \ref{rzstable} $r$ is stable. But now $r(x-\1,y,z)$ is a stable amalgam of $p$ and $q$ by Remark \ref{stableobvious}(d). This is a contradiction to Example \ref{StableEx}.
\end{exm}


\section{Real zero amalgamation conjectures}\label{hope}

Poljak and Turzík \cite{pt1} introduced the following notion for matroids.

\begin{df}\label{dfstickymatroid}
A matroid $M$ is called \emph{sticky} if whenever it is the common restriction of two matroids $M_1$ and $M_2$, then $M_1$ and $M_2$ can be amalgamated.
\end{df}

Inspired by this, we make the following definition.

\begin{df}
We call a real zero polynomial $f\in\R[x]$ \emph{sticky} if the Real Zero Amalgamation Problem \ref{rzap} admits a solution whenever $p(x,0)=q(x,0)=f$.
\end{df}

In this section, we will report what is known about sticky matroids and we will speculate about which real zero polynomials might be sticky.
We need some further standard notions for matroids first which are based on Definition \ref{dfclosure} above \cite[Pages 28 and 228]{oxl}.

\begin{df} Let $M=(S,\I)$ be a matroid.
\begin{enumerate}[(a)]
\item A set $F\subseteq S$ is called a \emph{flat} (or \emph{closed}) if $\cl(F)=F$.
\item $M$ is called \emph{modular} if $r(F\cap G)+r(F\cup G)=r(F)+r(G)$ for all flats $F,G\subseteq S$
\end{enumerate}
\end{df}

Note that the condition in (b) says that the submodularity inequality from Proposition \ref{rankproperties}(c) becomes sharp on flats.
Poljak and Turzík showed already that every modular matroid is sticky, that every sticky matroid of rank at most $3$ is modular and conjectured that
in fact every sticky matroid is modular \cite{pt1}. This became known as the \emph{sticky matroid conjecture} and was open almost for 40 years until
it got recently solved by Shin \cite{shi}.

\begin{thm}[Shin]
The sticky matroids are exactly the modular ones.
\end{thm}

The proof of Poljak and Turzík that every modular matroid is sticky is not very hard and it is a very easy exercise for the reader
that matroids on two-element sets are modular. So we record the following corollary.

\begin{cor}[Poljak and Turzík]
Matroids with a two-element ground set are sticky.
\end{cor}

A matroid-based example of the kind we have constructed in Section \ref{realzerocounter} above does therefore not exist for $\ell=2$ shared variables.
In addition, real zero polynomials in two variables are very special since they enjoy for example having the determinantal representation guaranteed by the
Helton-Vinnikov theorem \ref{vinnikov}. This gives hope that real zero polynomials in two variables could be sticky. Since this would have positive
consequences for the famous Generalized Lax Conjecture \cite[Subsection 6.1]{hv} mentioned at the end of Section \ref{stableRZHyp} \cite{sc2}, we
think that the following conjecture is well-motivated.

\begin{con}[Weak real zero amalgamation conjecture]
Real zero polynomials in two variables are sticky.
\end{con}

A stronger form of this conjecture motivated by our degree preserving approaches to real amalgamation in Theorem \ref{amalgamate} above is the following.

\begin{con}[Strong real zero amalgamation conjecture]
Let $\ell=2$, i.e., $x=(x_1,x_2)$, and $d\in\N_0$. Suppose
$p\in\R[x,y]$ and $q\in\R[x,z]$ are real zero polynomials of degree at most $d$ such that $p(x,0)=q(x,0)$. Then there exists a real zero polynomial $r\in\R[x,y,z]$ of degree at most $d$ such that \[p=r(x,y,0)\qquad\text{and}\qquad q=r(x,0,z).\]
\end{con}

\section*{Acknowledgements}
We thank Mario Bauer and Cynthia Vinzant for useful discussions at the Oberwolfach workshop \emph{Real Algebraic Geometry with a View toward Koopman Operator Methods}
in March 2023. We are very grateful to both anonymous referees for their help in improving the paper. In particular, we owe the now really simple proof of
Lemma \ref{walshreinterpret} to one of these referees.

\end{document}